\documentclass[graybox]{svmult}

\usepackage{type1cm}        
\usepackage{makeidx}         
\usepackage{graphicx}        
\usepackage{multicol}        
\usepackage[bottom]{footmisc}

\usepackage{newtxtext}       %
\usepackage{newtxmath}       

\makeindex             

\usepackage{xypic}

\usepackage{graphicx}
\usepackage{amsmath}
\usepackage{xcolor}

\renewcommand{\appendixname}{Appendix}

\newtheorem{theorem*}{Theorem}[subsection]
\newtheorem{lemma*}{Lemma}[subsection]
\newtheorem{definition*}{Definition}[subsection]

\usepackage{enumerate}

\let\e=\varepsilon

\let\p=\partial

\let\O=\Omega

\let\hide\iffalse
\let\unhide\fi

\newcommand{\R}{\mathbb{R}}

\renewcommand{\S}{\mathbb{S}}

\newcommand{\be}{\begin{equation}}
\newcommand{\bm}{\begin{multline}}
\newcommand{\ee}{\end{equation}}
\newcommand{\dd}{\mathrm{d}}

\newcommand{\xb}{x_{\mathbf{b}}}

\newcommand{\tb}{t_{\mathbf{b}}}
\newcommand{\vb}{v_{\mathbf{b}}}

\newcommand{\tbpm}{t_{\mathbf{b},\iota}}

\newcommand{\vbpm}{v_{\mathbf{b},\iota}}

\newcommand{\xbpm}{x_{\mathbf{b},\iota}}

\newcommand{\Bes}{\begin{eqnarray*}}
\newcommand{\Ees}{\end{eqnarray*}}
\newcommand{\Be}{\begin{equation} }
\newcommand{\Ee}{\end{equation}}
\newcommand{\Bs}{\begin{split}}
\newcommand{\Es}{\end{split}}

\def\p{\partial}

\def\O{\Omega}
\def\R{\mathbb{R}}

\def\B{\begin{equation}}
\def\E{\end{equation}}
\def\BN{\begin{eqnarray*}}
\def\EN{\end{eqnarray*}}

\begin{document}

\title*{On Some Recent Progress in the Vlasov-Poisson-Boltzmann System with Diffuse Reflection Boundary}
\titlerunning{The Vlasov-Poisson-Boltzmann system with Diffuse reflection boundary} 
\author{Yunbai Cao and Chanwoo Kim}
\institute{Yunbai Cao \at  Department of Mathematics, University of Wisconsin-Madison, Madison, WI, 53706, USA, \email{ycao35@wisc.edu}
\and Chanwoo Kim \at  Department of Mathematics, University of Wisconsin-Madison, Madison, WI, 53706, USA \email{ckim.pde@gmail.coms}}
%
%
\maketitle

\abstract*{Each chapter should be preceded by an abstract (no more than 200 words) that summarizes the content. The abstract will appear \textit{online} at \url{www.SpringerLink.com} and be available with unrestricted access. This allows unregistered users to read the abstract as a teaser for the complete chapter.
Please use the 'starred' version of the \texttt{abstract} command for typesetting the text of the online abstracts (cf. source file of this chapter template \texttt{abstract}) and include them with the source files of your manuscript. Use the plain \texttt{abstract} command if the abstract is also to appear in the printed version of the book.}

\abstract{We discuss some recent development on the Vlasov-Poisson-Boltzmann system in bounded domains with diffuse reflection boundary condition. In addition we present a new regularity result when the particles are surrounded by conductor boundary.}

\section{Some recent development}
\subsection{Background}

The object of kinetic theory is the modeling of particles by a distribution function in the phase space, which is denoted by $F(t,x,v)$ for $(t,x,v) \in [0, \infty) \times  {\O} \times \R^{3}$ where $\O$ is an open bounded subset of $\R^{3}$. Dynamics and collision processes of dilute charged particles with an electric field $E$ can be modeled by the (two-species) Vlasov-Poisson-Boltzmann equation
\Be \label{2FVPB}
\begin{split}
\p_t F_+ + v \cdot \nabla_x F_+ +E \cdot \nabla_v F_+ = Q(F_+,F_+) + Q(F_+,F_- ),
\\ \p_t F_- + v \cdot \nabla_x F_- -E \cdot \nabla_v F_- = Q(F_-,F_+) + Q(F_-,F_- ).
\end{split} \Ee
Here $F_\pm(t,x,v) \ge 0 $ are the density functions for the ions $(+)$ and electrons $(-)$ respectively.


The collision operator measures ``the change rate'' in binary hard sphere collisions and takes the form of
\Be\begin{split}\label{Q}
Q(F_{1},F_{2}) (v)&: = Q_\mathrm{gain}(F_1,F_2)-Q_\mathrm{loss}(F_1,F_2)\\
&: =\int_{\R^3} \int_{\S^2} 
|(v-u) \cdot \omega| [F_1 (u^\prime) F_2 (v^\prime) - F_1 (u) F_2 (v)]
 \dd \omega \dd u,
\end{split}\Ee   
where $u^\prime = u - [(u-v) \cdot \omega] \omega$ and $v^\prime = v + [(u-v) \cdot \omega] \omega$. The collision operator enjoys a collision invariance: for any measurable $G$,  $
\int_{\R^{3}} \begin{bmatrix}1 & v & \frac{|v|^{2}-3}{2}\end{bmatrix} Q(G,G) \dd v = \begin{bmatrix}0 & 0 & 0 \end{bmatrix}.$ It is well-known that a global Maxwellian $\mu$ 
satisfies $Q(\cdot,\cdot)=0$, where
\Be\label{Maxwellian}
\mu(v):= \frac{1}{(2\pi)^{3/2}} \exp\bigg(
 - \frac{|v |^{2}}{2 }
 \bigg).
\Ee


The electric field $E$ is given by
\Be\label{Field}
E(t,x): =  - \nabla_{x} \phi(t,x),
\Ee
where an electrostatic potential is determined by the Poisson equation: 
\Be  \label{Poisson2}
-\Delta_x \phi(t,x) = \int_{\mathbb R^3 } (F_+(t,x,v) - F_-(t,x,v))  \, dv
\ \  \text{in} \ \O .
\Ee

A simplified one-species Vlasov-Poisson-Boltzmann equation is often considered to reduce the complexity. Where we let $F(t,x,v)$ takes the role of $F_+(t,x,v)$, and assume $F_- = \rho_0 \mu $ where the constant $\rho_0 = \int_{\O \times \mathbb R^3} F_+(0,x,v)  \,dv dx$. Then we get the system
\Be \label{Boltzmann_E}
\partial_{t} F + v\cdot \nabla_{x} F + E\cdot \nabla_{v} F = Q(F,F),
\Ee 
\Be  \label{Poisson}
-\Delta_x \phi(t,x) = \int_{\mathbb R^3 } F(t,x,v)  \, dv -\rho_0
\ \  \text{in} \ \O .
\Ee
Here the background charge density $\rho_0$ is assumed to be a constant.

Throughout this paper, we use the notation
\Be \label{iota} \begin{split}
 \iota = +  \text{ or } -, \text{ and }
 -\iota = \begin{cases}
- &, \text{ if } \iota = + \\ +&, \text{ if } \iota = -.
\end{cases}
\end{split} \Ee
And for the one-species case, $F_\iota = F$.


In many physical applications, e.g. semiconductor and tokamak, the charged dilute
gas is confined within a container, and its interaction with the boundary, which can be described by suitable boundary conditions, often plays a crucial role in global dynamics. In this paper we consider one of the physical conditions, a so-called diffuse boundary condition:
\Be\label{diffuse_BC}
F_\iota(t,x,v)= \sqrt{2\pi} \mu(v) \int_{n(x ) \cdot u>0} F_\iota(t,x,u) \{n(x) \cdot u\}\dd u \ \ \text{for} \ (x,v) \in \gamma_-.
\Ee
Here $\gamma_-: = \{(x,v) \in \p\O \times \R^3: n(x) \cdot v<0\}$, and $n(x)$ is the outward unit normal at a boundary point $x$. 

%

Due to its importance, there have been many research activities in mathematical study of the Boltzmann equation. In \cite{Guo_P}, global strong solution of Boltzmann equation coupled with the Poisson equation has been established through the nonlinear energy method, when the initial data are close to the Maxwellian $\mu$. In the large-amplitude regime, an almost exponential decay for Boltzmann solutions is established in \cite{DV}, provided certain a priori strong Sobolev estimates can be verified. Such high regularity insures an $L^\infty$-control of solutions which is crucial to handle the quadratic nonlinearity. Even though these estimates can be verified in periodic domains, their validity in general bounded domains have been doubted. 

Despite its importance, mathematical theory on boundary problems of VPB, especially for strong solutions, hasn't been developed up to satisfactory (cf. renormalized solutions of VPB were constructed in \cite{Michler}). One of the fundamental difficulties for the system in bounded domains is the lack of higher regularity, which originates from the characteristic nature of boundary conditions in the kinetic theory, and the nonlocal property of the collision term $Q$. This nonlocal term indicates that the local behavior of the solution could be affected globally by $x$ and $v$, and thus prevents the localization of the solution. From that a seemingly inevitable singularity of the spatial normal derivative at the boundary $x \in \p \O$ arises
$\p_n F_\iota(t,x,v) \sim \frac{1 }{ n(x) \cdot v  }  \notin L^1_{loc}.$
Such singularity towards the grazing set $\gamma_0 : = \{(x,v) \in \p\O \times \R^3: n(x) \cdot v=0\}$ has been studied thoroughly in \cite{GKTT1} for the Boltzmann equation in convex domain. Here we clarify that a $C^{\alpha}$ domain means that for any ${p} \in \partial{\Omega}$, there exists sufficiently small $\delta_{1}>0, \delta_{2}>0$, and an one-to-one and onto $C^{\alpha}$-map, $
	\eta_p: \{( x_{\|,1}, x_{\|,2} , x_n ) \in \mathbb R^3 : x_n > 0 \} \cap B(0; \delta_1 ) \to \Omega \cap B(p;\delta_2 )$ with $\eta_p ( x_{\|,1}, x_{\|,2} , x_n ) = \eta_p( x_{\|,1}, x_{\|,2} , 0 ) + x_n [-n (\eta_p( x_{\|,1}, x_{\|,2} , 0))],$
	such that $\eta_p (\cdot, \cdot, 0) \in \p\O$. 
	A \textit{convex} domain means that there exists $C_\O>0$ such that for all $p \in \p\O$ and  $\eta_p$ and for all $x_\parallel$,
\begin{equation}\label{convexity_eta}
\begin{split}
\sum_{i,j=1}^{2} \zeta_{i} \zeta_{j}\p_{i} \p_{j} \eta _{{p}}   ( x_{\parallel }  )\cdot  
n ( x_{\parallel } )
  \leq    - C_{\Omega} |\zeta|^{2}  \ 
  \text{ for all}   \ \zeta \in \mathbb{R}^{2}.
\end{split}
\end{equation}


Construction of a unique global solution and proving its asymptotic stability of VPB in general domains has been a challenging open problem for any boundary condition. In \cite{VPB} the authors give the \textit{first} construction of a unique global \textit{strong} solution of the one-species VPB system with the diffuse boundary condition when the domain is $C^3$ and \textit{convex.} Moreover an asymptotic stability of the global Maxwellian $\mu$ is studied. The result was then extended to the two-species case in \cite{2SVPB}.



%

\subsection{Global strong solution of VPB}
In \cite{VPB, {2SVPB}}, the authors take the first step toward comprehensive understanding of VPB in bounded domains. They consider the zero Neumann boundary condition for the potential $\phi$: $n \cdot E \vert_{\p \O }= \frac{ \p \phi }{\p n } \vert_{\p \O } = 0$, which corresponds to a so-called insulator boundary condition. In such setting $(F_\iota, E) = (\mu, 0)$ is a stationary solution. 

The characteristics (trajectory) is determined by the Hamilton ODEs for $f_+$ and $f_-$ separately
\Be\label{hamilton_ODE1}
\frac{d}{ds} \left[ \begin{matrix}X_\iota^f(s;t,x,v)\\ V_\iota^f(s;t,x,v)\end{matrix} \right] = \left[ \begin{matrix}V_\iota^f(s;t,x,v)\\ 
{-\iota} \nabla_x \phi_f
(s, X_\iota^f(s;t,x,v))\end{matrix} \right]  \ \ \text{for}   - \infty< s ,  t < \infty  ,
\Ee
with $(X_\iota^f(t;t,x,v), V_\iota^f(t;t,x,v)) =  (x,v)$. Where the potential is extended to negative time as $\phi_f(t,x)= e^{-|t|} \phi_{f_0}(x)$ for $t\leq 0$.


For $(t,x,v) \in \R  \times  \O \times \R^3$, define \textit{the backward exit time} $\tbpm^f(t,x,v)$ as   
\Be\label{tb}
\tbpm^f (t,x,v) := \sup \{s \geq 0 : X_\iota^f(\tau;t,x,v) \in \O \ \ \text{for all } \tau \in (t-s,t) \}.
\Ee
Furthermore, define $\xbpm^f (t,x,v) := X_\iota^f(t-\tbpm(t,x,v);t,x,v)$ and $\vbpm^f (t,x,v) := V_\iota^f(t-\tbpm(t,x,v);t,x,v)$.


In order to handle the boundary singularity, they introduce the following notion
\begin{definition}[Kinetic Weight] \label{kweight} For $\e>0$
\Be\label{alphaweight}\begin{split}
\alpha_{f,\e,\iota}(t,x,v) : =& \  
\chi \Big(\frac{t-\tbpm^{f}(t,x,v)+\e}{\e}\Big)
|n(\xbpm^{f}(t,x,v)) \cdot \vbpm^{f}(t,x,v)| \\
&+ \Big[1- \chi \Big(\frac{t-\tbpm^{f}(t,x,v) +\e}{\e}\Big)\Big].
\end{split}\Ee
Here they use a smooth function $\chi: \R \rightarrow [0,1]$ satisfying
\Be\label{chi}
\begin{split}
\chi(\tau)  =0,  \     \tau\leq 0, \ \text{and} \  \ 
\chi(\tau)  = 1    ,  \  \tau\geq 1. 
 \ \ \frac{d}{d\tau}\chi(\tau)  \in [0,4] \ \   \text{for all }   \tau \in \R.
\end{split}
\Ee
\end{definition}
Also, denote
\Be \label{matrixalpha}
\alpha_{f,\e}(t,x,v) := \begin{bmatrix} \alpha_{f, \e, +}(t,x,v) & 0 \\ 0 & \alpha_{f, \e, -}(t,x,v) \end{bmatrix}.
\Ee

Note that $\alpha_{f,\e,\iota}(0,x,v)\equiv \alpha_{{f_0},\e,\iota}(0,x,v)$ is determined by $f_0$. For the sake of simplicity,  the superscription $^f$ in $X_\iota^f, V_\iota^f, \tbpm^f, \xbpm^f, \vbpm^f$ is dropped unless they could cause any confusion.

One of the crucial properties of the kinetic weight in (\ref{alphaweight}) is an invariance under the Vlasov operator: $
\big[\p_t + v\cdot \nabla_x - \nabla_x \phi_f \cdot \nabla_v \big] \alpha_{f,\e,\iota}(t,x,v) =0.$ This is due to the fact that the characteristics solves a deterministic system (\ref{hamilton_ODE1}). This crucial invariant property under the Vlasov operator is one of the key points in their approach in \cite{VPB, 2SVPB}.

Denote $
\label{weight}
w_\vartheta(v) =  e^{\vartheta|v|^2}.$

\begin{theorem}[\cite{VPB, 2SVPB}] 
\label{main_existence}
Assume a bounded open $C^3$ domain $\O \subset\R^3$ is convex (\ref{convexity_eta}). Let $0< \tilde{\vartheta}< \vartheta\ll1$. Assume the compatibility condition: (\ref{diffuse_BC}) holds at $t=0$. 
There exists a small constant $0< \e_0 \ll 1$ such that for all $0< \e \leq \e_0$ if an initial datum $F_{0,\iota} =\mu + \sqrt \mu f_{0,\iota}$ satisfies
\Be\label{small_initial_stronger}
 \|w_\vartheta f_{0,\iota} \|_{L^\infty(\bar{\O} \times \R^3)}< \e,  \|   w_{\tilde{\vartheta}}   \nabla_{v } f_{0,\iota} \|_{ {L}^{3 } ( {\O} \times \R^3)}< \infty,
 \Ee
\Be\label{W1p_initial}
\begin{split}
 \| w_{\tilde{\vartheta}} \alpha_{f_{0,\iota}, \e }^\beta \nabla_{x,v } f_{0,\iota} \|_{ {L}^{p } ( {\O} \times \R^3)}
 <\e
\ \
\text{for}  \  \ 3< p < 6, \ \ 
1-\frac{2}{p }
 < \beta<
\frac{2}{3}
,\end{split}
\Ee
then there exists a unique global-in-time solution $F_\iota(t)= \mu+ \sqrt{\mu} f_\iota(t) \geq 0$ to (\ref{2FVPB}), (\ref{Field}), (\ref{Poisson2}), \eqref{diffuse_BC}. Moreover there exists $\lambda_{\infty} > 0$ such that 
\Be\begin{split}\label{main_Linfty}
 \sup_{ t \geq0}e^{\lambda_{\infty} t} \| w_\vartheta f_\iota(t)\|_{L^\infty(\bar{\O} \times \R^3)}+ 
 \sup_{ t \geq0}e^{\lambda_{\infty} t} \| \phi_f(t)  \|_{C^{2}(\O)}  \lesssim 1,
\end{split}\Ee
and, for some $C>0$, and, for $0< \delta= \delta(p,\beta) $,
\Be\label{W1p_main}
 \| w_{\tilde{\vartheta}} \alpha_{f, \e ,\iota }^{\beta } \nabla_{x,v} f_\iota(t)  \|_{L^{ p} ( {\O} \times \R^3)} 
 \lesssim e^{Ct} \ \ \text{for all } t \geq 0
,
\Ee
\Be\label{nabla_v f_31}
\| \nabla_v f_\iota (t) \|_{L^3_x (\O) L^{1+\delta }_v (\R^3)} \lesssim_t 1  \ \ \text{for all } \  t\geq 0.
\Ee

Furthermore, if $F_\iota$ nad $G_\iota$ are both solutions to (\ref{2FVPB}), (\ref{Field}), (\ref{Poisson2}), \eqref{diffuse_BC},
then 
\Be\label{stability_1+}
\| f_\iota(t) - g_\iota(t) \|_{L^{1+\delta} (\O \times \R^3)} \lesssim_t \| f_\iota(0) - g_\iota(0) \|_{L^{1+\delta} (\O \times \R^3)} \ \ \text{for all } \  t\geq 0.
\Ee 

\end{theorem}
\begin{remark}The second author and his collaborators constructs a local-in-time solution for given general large datum in \cite{CKL} for the generalized diffuse reflection boundary condition. By introducing a scattering kernel $R(u \rightarrow v;x,t)$, representing the probability of a molecule striking in the boundary at $x\in\partial\Omega$ with velocity $u$ to be bounced back to the domain with velocity $v$, they consider  \begin{equation}\begin{split}\label{eqn:BC}
&F(t,x,v) |n(x) \cdot v|= \int_{\gamma_+(x)}
R(u \rightarrow v;x,t) F(t,x,u)
\{n(x) \cdot u\} d u, \quad \text{ on }\gamma_-
.
\end{split}
\end{equation}
 In \cite{CKL} they study a model proposed by Cercignani and Lampis in~\cite{CIP,CL}. With two accommodation coefficients $  0<r_\perp\leq 1,\quad 0<r_\parallel<2 ,$
%
the Cercignani-Lampis boundary condition (C-L boundary condition) can be written as
\begin{equation}\label{eqn: Formula for R}\begin{split}
 &R(u \rightarrow v;x,t)\\
:=&  \frac{1}{r_\perp r_\parallel (2- r_\parallel)\pi/2} \frac{|n(x) \cdot v|}{(2T_w(x))^2}
I_0 \left(
 \frac{1}{2T_w(x)}\frac{2 (1-r_\perp)^{1/2} v_\perp u_\perp}{r_\perp}
\right)
\\
& \times 
\exp\left(- \frac{1}{2T_w(x)}\left[
\frac{|v_\perp|^2 + (1- r_\perp) |u_\perp|^2}{r_\perp}
+ \frac{|v_\parallel - (1- r_\parallel ) u_\parallel|^2}{r_\parallel (2- r_\parallel)}
\right]\right).
\end{split}
\end{equation}
Here $T_w(x)$ is a wall temperature on the boundary and $I_0 (y) := \pi^{-1} \int^{\pi}_0e^{y \cos \phi } d \phi$.
In this formula, $v_\perp$ and $v_\parallel$ denote the normal and tangential components of the velocity respectively: $   v_\perp= v\cdot n(x) ,  v_\parallel = v- v_\perp n(x)\,$.
\end{remark}

In \cite{VPB, 2SVPB} a global $L^\infty$-bound is proven by $L^2-L^\infty$ framework. The idea is to use Duhamel's principle to estimate the solution $f$ along the characteristics \eqref{hamilton_ODE1} to reach
\Be \begin{split}\notag
   \| f_\iota(t) \|_{L^\infty(\bar \O \times \mathbb R^3)}
 \sim  \| e^{-  t} f_{0,\iota} \|_{L^\infty(\bar \O \times \mathbb R^3)} + \int_0^t e^{- (t-s ) } \| f_\iota(s) \|_{L^2(\O \times \mathbb R^3) } ds.
\end{split} \Ee
And then use the decay of $f_\iota$ in $L^2$ norm to conclude the decay in $L^\infty$.
The key of this process is to verify
\Be\begin{split}
\frac{ \p X_\iota(s;t,x,v) }{ \p v } &\sim -(t-s)\text{Id}_{3 \times 3 } + \int_s^t \int_{s'}^t \nabla_x^2 \phi(s'') \frac{ \p X_\iota (s'';t,x,v) }{\p v } ds' ds'' \\
&\sim O(|t-s|) \text{Id}_{3\times3}.
\end{split}\Ee
For which the $C^2$-bound of $\phi$ seems necessary. Unfortunately such $C^2$ estimate for $\phi$ falls short of the boarder line case of the Schauder elliptic regularity theory when the source term of the Poisson equation $\int_{\mathbb R^3 } (F_+ - F_- )  dv $ in \eqref{Poisson2} is merely continuous or bounded. They overcome such difficulty by interpolating the $C^2$ norm into a sum of a $C^{2,0+}$ norm and a $C^{1,1-}$ norm:
\begin{lemma}\label{lemma_interpolation}Assume $\O \subset \R^3$ with a $C^2$ boundary $\p\O$. For $0< D_1<1$, $0< D_2<1$, and $\Lambda_0>0$,
		\Be\begin{split}\label{phi_interpolation}
			\|\nabla^2_x \phi(t )\|_{L^\infty (\O)}
			\lesssim_{\O, D_1, D_2}
			e^{D_1 \Lambda_0t}\|  \phi(t)\|_{C^{1,1-D_1}(\O)}
			+ e^{- D_2 \Lambda_0t}\|  \phi(t)\|_{C^{2, D_2}(\O)}.
		\end{split}\Ee
	\end{lemma} 
 While an exponential decay of the weaker $C^{1,1-}$ norm can be derived from the exponential decay of $f_\iota$ in $L^\infty$, the $C^{2,0+}$ norm is controlled by Morrey's inequality
\Be \label{phic2pbd}
\| \phi \|_{C_x^{2,0 +}} \lesssim \sum_{\iota = \pm} \| \int_{\mathbb R^3 } f\iota \sqrt \mu dv \|_{C_x^{0,0+} } \lesssim  \sum_{\iota = \pm} \| \int_{\mathbb R^3} \nabla_x f_\iota \sqrt \mu dv \|_{L^p_x} , \text{ for } p>3.
\Ee

Now the spatial derivative of $f_\iota$ needs to be controlled.
%
They develop an $ \alpha_\iota$-weighted $W^{1,p}$ estimate by energy-type estimate of $ \alpha_\iota \nabla_{x,v} f_\iota$, where the $ \alpha_\iota$-multiplication compensates the boundary singularity. This allows us to bound \eqref{phic2pbd} for $\frac{p-2}{p} < \beta < \frac{p-1}{p}$,
\[
 \| \int_{\mathbb R^3} \nabla_x f_\iota \sqrt \mu dv \|_{L^p_x} \lesssim \|  \alpha_\iota^{-\beta} \|_{L^{\frac{p}{p-1}}} \|  \alpha_\iota^\beta \nabla_x f_\iota \sqrt \mu \|_{L^p_{x,v} } \lesssim \|  \alpha_\iota^\beta \nabla_x f_\iota \sqrt \mu \|_{L^p_{x,v} },
\]
as long as
\Be \label{int1overalpha}
 \alpha_\iota^{-\frac{\beta p}{p-1} } \sim \frac{1}{ \alpha_\iota(t,x,v)^{1-}} \in L^1_v \text{ uniformly for all } x.
\Ee

A difficulty of the proof of \eqref{int1overalpha} arises form lack of local representation of $ \alpha_\iota (t, x, v)$. $\alpha_\iota$ is only defined at some boundary point along (possibly very complicated) characteristics. They employ a geometric change of variables $v \mapsto (\xbpm(t,x,v),\tbpm(t,x,v) )$ to exam \eqref{int1overalpha}. By computing the Jacobian there is an extra $ \alpha$-factor from $dv \sim \frac{  \alpha_\iota}{|\tbpm|^3 } d\tbpm d \xbpm$, which cancels the singularity of \eqref{int1overalpha}. Then they use a lower bound of $\tbpm \gtrsim \frac{|\xbpm^f-x|}{\max |V|}$ and a bound $ \alpha \lesssim \frac{|(x-\xbpm^f) \cdot n(\xbpm^f)|}{\tbpm^f}$ to have
\Be\label{alpha_bounded_intro}
\int_{|v| \lesssim 1} { \alpha_\iota}^{-  \frac{\beta p}{p-1}} \dd v \lesssim \int_{\text{boundary}} \frac{|(x- \xbpm) \cdot n(\xbpm)|^{1-  \frac{\beta p}{p-1}}}{|x-\xbpm|^{3-  \frac{\beta p}{p-1}}} \dd \xbpm
+ \text{good terms}< \infty, 
\Ee
which turns to be bounded as long as $ \frac{\beta p}{p-1}<1$.


From the above estimates and the interpolation, they derive \textit{an exponential decay} of $\phi(t)$ in $C_x^2$ as long as $\|  \alpha_\iota^\beta \nabla_x f(t) \|_{L_{x,v}^p} $ grows at most exponentially. With the $C_x^2$-bound of $\phi$ in hand, they control $\|  \alpha_\iota^\beta \nabla_x f(t) \|_{L_{x,v}^p} $via Gronwall's inequality and close the estimate by proving its (at most) exponential growth.

For the uniqueness and stability of approximating sequence they prove $L^1$-stability. The key observation is that $v$-derivatives of the diffuse BC \eqref{diffuse_BC} has no boundary singularity, thus is bounded. The equation of $\nabla_v f_\iota$ has a singular forcing term $\nabla_x f_\iota$. 
For which they control $\| \nabla_x f_\iota \|_{L_x^3 L_v^1 } $ as $\|  \alpha_\iota^{-\beta} \|_{L_v^{\frac{p}{p-1}} } \|  \alpha_\iota^\beta \nabla_x f_\iota\|_{L^p_{x,v} } $, and this term is bounded from \eqref{int1overalpha}.

\hide
For the two-species VPB system in \cite{2SVPB}, when performing the expansion for $\begin{bmatrix} F_+ \\ F_- \end{bmatrix} = \begin{bmatrix} \mu + \sqrt \mu f_+ \\ \mu + \sqrt \mu f_- \end{bmatrix}$, the vector linearized Boltzmann operator $L$ is defined as
\Be \label{L_decomposition}
L \begin{bmatrix} g_1 \\ g_2 \end{bmatrix} :=  -\frac{1}{\sqrt \mu } \begin{bmatrix}  2 Q(\sqrt \mu g_1, \mu ) + Q (\mu, \sqrt \mu ( g_1 + g_2 ) )
\\   2 Q(\sqrt \mu g_2, \mu ) + Q (\mu, \sqrt \mu ( g_1 + g_2 ) )
 \end{bmatrix}.
 \Ee
The null space of $L$ is a six-dimensional subspace of $L^2_v(\mathbb R^3; \mathbb R^2 )$ spanned by orthonormal vectors
\Be 
\left\{ \begin{bmatrix} \sqrt \mu \\ 0  \end{bmatrix}, \begin{bmatrix} 0  \\ \sqrt \mu \end{bmatrix}, \begin{bmatrix} \frac{v_i}{\sqrt 2 } \sqrt \mu \\ \frac{v_i}{\sqrt 2 } \sqrt \mu   \end{bmatrix}, \begin{bmatrix} \frac{|v|^2 - 3}{2\sqrt 2} \sqrt \mu \\ \frac{|v|^2 - 3}{2\sqrt 2} \sqrt \mu   \end{bmatrix}
 \right\}, \, i = 1,2,3.
\Ee
And the projection of $\mathbf f = \begin{bmatrix} f_+ \\ f_- \end{bmatrix}$ onto the null space $N(L)$ can be denoted by
\Be \begin{split}
& \mathbf P \mathbf f (t,x,v)
\\ & := \left\{ a_+(t,x) \begin{bmatrix} \sqrt \mu \\ 0  \end{bmatrix} + a_-(t,x) \begin{bmatrix} 0  \\ \sqrt \mu \end{bmatrix} + b(t,x)  \cdot \frac{v}{\sqrt 2 } \begin{bmatrix} \sqrt \mu \\ \sqrt \mu  \end{bmatrix} + c(t,x)  \frac{|v|^2 - 3}{2\sqrt 2}\begin{bmatrix} \sqrt \mu \\ \sqrt \mu  \end{bmatrix}
\right\}.
\end{split} \Ee
Using the standard $L^2$ energy estimate of the equation, it is well-known that $L$ is degenerate: $\left \langle L \mathbf f , \mathbf f \right \rangle  \gtrsim \left \|  (I - \mathbf P) \mathbf f \right \|_{L^2_{x,v }}$. Thus it's clear that in order to control the $L^2$ norm of $f_\iota(t)$, a way to bound the missing $ \left \| \mathbf P\mathbf f (t) \right \|_{L^2} $ term is needed. Fortunately, the technique of test function method developed in \cite{EGKM} can be applied to the two-species settings. By the weak formulation of the equation and a set of properly choosing test functions, the $a_\pm(t,x), b(t,x), c(t,x)$ can be controlled. And the estimate of $\mathbf P\mathbf f $  is achieved as
\[
\left \| \mathbf P\mathbf f (s) \right \|_{L_{x,v}^2 } \lesssim \left \| ( I - \mathbf P ) \mathbf f \right \|_{L_{x,v}^2 } + \text{ "good terms"}.
\]
\unhide

%

\subsection{Improved regularity under the sign condition}
One interesting question is to improve the regularity estimate beyond a weighted $W^{1,p}$ for $p< 6$ of $f_\iota$ in \cite{VPB, 2SVPB}. Some work in this direction has been done in \cite{VPBEP}.

In \cite{VPBEP} the author consider the one-species VPB system \eqref{Field}, \eqref{Boltzmann_E}, where the potential consists of a self-generated electrostatic potential and an external potential. That is $E = \nabla \phi$, where
\begin{equation} \label{VPB2}
\phi(t,x) = \phi_F(t,x) + \phi_E(t,x), \text{ with } \frac{\p \phi_E}{\p n } > C_E > 0 \text{ on } \p \Omega,
\end{equation}
and $\phi_F$ satisfies \eqref{Poisson} and the zero Neumann boundary condition $\frac{ \p \phi_F}{\p n } = 0$ on $\p \O$. Under such setting, the field $E$ satifies a crucial sign condition on the boundary
\Be \label{signEonbdry}
E(t,x) \cdot n(x) > C_E > 0 \text{ for all } t \text{ and all } x \in \p \O.
\Ee

With the help of the external potential $\phi_E$ with the crucial sign condition \eqref{VPB2}, they construct a short time weighted $W^{1,\infty}$ solution to the VPB system, which improves the regularity estimate of such system in Theorem \ref{main_existence}. The key idea of the result is to incorporate a different distance function $\tilde \alpha$:
\Be \label{alphatilde}
\tilde \alpha \sim \bigg[ |v \cdot \nabla \xi (x)| ^2 + \xi (x)^2 - 2 (v \cdot \nabla^2 \xi(x) \cdot v ) \xi(x) - 2(E(t,\overline x ) \cdot \nabla \xi (\overline x ) )\xi(x) \bigg]^{1/2},
\Ee
where $\xi:\mathbb R^3 \to \mathbb R$ is a smooth function such that $ \Omega = \{ x \in \mathbb R^3: \xi(x) < 0 \}$, and the closest boundary point $\overline x := \{ \bar x \in \p \Omega :  d(x,\bar x ) = d(x, \partial \Omega) \}$ is uniquely defined for $x$ closed to the boundary. Note that $\tilde \alpha \vert_{\gamma_-} \sim | n(x) \cdot v |$. A version of a distance function without the potential was used in \cite{GKTT1}. One of the key contribution in \cite{VPBEP} is to incorporate this different distance function \eqref{alphatilde} in the presence of an external field.

\begin{theorem}[\cite{VPBEP}]\label{WlinftyVPBthm} Let $\phi_E (t,x)$ be a given external potential with $\nabla_x \phi_E$ satisfying (\ref{signEonbdry}), and 
$ \| \nabla_x \phi_E(t,x) \|_{C^1_{t,x}(\R_+ \times\bar \O )} 
 < \infty.$
 Assume that, for some $0< \vartheta < \frac{1}{4}$,
$\|  w_\vartheta \tilde{\alpha} \nabla_{x,v} f_0 \|_{L^\infty(\bar \O \times \mathbb R^3)} + \| w_\vartheta f_0 \|_{L^\infty(\bar \O \times \mathbb R^3)} 
< \infty.$
Then there exists a unique solution $F(t,x,v) = \sqrt \mu f(t,x,v) $ to (\ref{Boltzmann_E}), \eqref{Field},  (\ref{diffuse_BC}), (\ref{VPB2}) for $t \in [0,T]$ with $0 < T \ll 1$, such that for some $0< \vartheta' < \vartheta $, $\varpi \gg 1$,
$\sup_{0\le t \le T} \| w_{\vartheta'} f(t) \|_{L^\infty(\bar \O \times \mathbb R^3)} < \infty,$ and
\begin{equation}
\sup_{0 \le t \le T}  \| w_{\vartheta'} e^{-\varpi  \langle v \rangle t } \tilde{\alpha}  \nabla_{x,v}  f^{} (t,x,v) \|_{L^\infty(\bar \O \times \mathbb R^3)} < \infty. 
\end{equation}
\end{theorem}

One of the crucial property $\tilde \alpha$ enjoys, under the assumption of the sign condition \eqref{signEonbdry}, is the invariance along the characteristics:

\begin{lemma}[Velocity lemma near boundary] \label{velocitylemma} 
Suppose $E(t,x)$ satisfies the sign condition (\ref{signEonbdry}).
Then for any $0 \le s<t $ and trajectory $X(\tau), V(\tau)$ solving (\ref{hamilton_ODE1}), if $X(\tau ) \in \Omega $ for all $s \le \tau \le t $, then
\begin{equation}
\begin{split}  \label{vlemma}
e^{ - C \int_s ^ t ( |V(\tau')| + 1 ) d \tau'} \tilde \alpha ( s,X(s),V(s) ) &\le \tilde \alpha (t,X(t),V(t)) \\&\le  e^{C \int_s ^ t ( |V(\tau')| + 1 ) d \tau'} \tilde \alpha (s,X(s),V(s)),
\end{split}\end{equation}
for any $C \gtrsim  (   \| \nabla_x \phi_E(t,x) \|_{C^1_{t,x}(\R_+ \times\bar \O )} + 1  )/{C_E} $.
\end{lemma}

The key ingredient in the $\tilde \alpha$-weighted regularity estimate is a dynamical non-local to local estimate which can be stated as
\begin{lemma}\label{keylemma}
Let $(t,x,v) \in [0,T] \times  \Omega \times \mathbb R^3$, $ 1 < \beta < 3$, $0 < \kappa \le 1 $. Suppose $E$ satisfies the sign condition (\ref{signEonbdry}). Then for $\varpi \gg 1$ large enough, and for any $0< C_\vartheta < \frac{1}{4}$, $0 < \delta \ll 1$,
\begin{equation} \label{nonlocaltolocal}
 \begin{split}
 & \int_{\max\{ 0, t - \tb \}} ^t  \int_{\mathbb R^3} e^{ -\int_s^t\frac{\varpi}{2} \langle V(\tau;t,x,v) \rangle d\tau }  \frac{e^{-\frac{C_\vartheta}{2} |V(s)-u|^2 }}{|V(s) -u |^{2 - \kappa}  }  \frac{1}{(\tilde \alpha(s,X(s) ,u))^\beta} du  ds
\\  \lesssim &   e^{ 2C_\O \frac{\| \nabla E\|_\infty  + \| E \|_{L^\infty_{t,x}}^2 + \| E \|_{L^\infty_{t,x}}}{C_E}}  \frac{  \delta ^{\frac{3 -\beta}{2} }}{\langle v\rangle^2(C_E+1)^ {\frac{\beta-1}{2}} (\tilde \alpha(t,x,v))^{\beta -2 }  (   \| E \|_{L^\infty_{t,x}}^2 + 1 )^{\frac{3 -\beta}{2}}}  
\\ & +  \frac{ ( \| E \|_{L^\infty_{t,x}}^2+1)^{\beta -1} } { C_E^{\beta-1} \delta ^{\beta -1 } (\tilde \alpha ( t,x,v) )^{\beta - 1 } } \frac{2}{ \varpi },
\end{split} 
\end{equation}
where $(X(s), V(s)) = (X(s;t,x,v), V(s;t,x,v)) $ as in \eqref{hamilton_ODE1}. \end{lemma}
The same estimate without the external field had been established by the second author and his collaborators in \cite{GKTT1}. The proof of \eqref{nonlocaltolocal} is obtained by first making use of a series of change of variables to get the precise estimate of the velocity integration, which is bounded by,
\Be \label{kernelbddfor1overalpha}
\int_{\mathbb R^3} \frac{ e^{-\vartheta |V(s)- u|^2}}{ |V(s) - u | ^{2 - \kappa} [\tilde \alpha(s, X(s), u ) ]^\beta } du \lesssim \frac{ 1}{( |V(s)|^2 \xi(X(s)) - C_E  \xi (X(s)) )^{\frac{\beta - 1}{2} } } , 
\Ee
then followed by relating the time integration back to ${\tilde \alpha}^{-1}$. For the later part of the proof, the velocity lemma \eqref{vlemma} and the boundedness of the external field to ensure the monotonicity of $|\xi(X(s))|$ near the boundary, where the change of variable
$
dt \simeq \frac{d\xi}{|v\cdot \nabla \xi |},$
can be performed and recovers a power of $\tilde \alpha$ in the $\xi$-integration. On the other hand, the sign condition (\ref{VPB2}) is crucially used to establish a lower bound for $|\xi(X(s))|$ when it's away from the boundary, which helps to recover a power of $\tilde \alpha$ as wanted.

\hide
\subsection{Generalized Diffuse Boundary Condition}
In \cite{CKL} the authors study the one-species VPB system with generalized diffuse boundary condition.
The diffuse boundary condition \eqref{diffuse_BC} can be generalized by introducing a scattering kernel $R(u \rightarrow v;x,t)$ through a general balance law of
\begin{equation}\begin{split}\label{eqn:BC}
&F(t,x,v) |n(x) \cdot v|= \int_{\gamma_+(x)}
R(u \rightarrow v;x,t) F(t,x,u)
\{n(x) \cdot u\} d u, \quad \text{ on }\gamma_-
.
\end{split}
\end{equation}
Physically, $R(u\to v;x,t)$ represents the probability of a molecule striking in the boundary at $x\in\partial\Omega$ with velocity $u$ to be bounced back to the domain with velocity $v$ at the same location $x$ and time $t$. In \cite{CKL} they study a model proposed by Cercignani and Lampis in~\cite{CIP,CL}. With two accommodation coefficients $  0<r_\perp\leq 1,\quad 0<r_\parallel<2 ,$
%
the Cercignani-Lampis boundary condition (C-L boundary condition) can be written as
\begin{equation}\label{eqn: Formula for R}\begin{split}
&R(u \rightarrow v;x,t)\\
:=& \frac{1}{r_\perp r_\parallel (2- r_\parallel)\pi/2} \frac{|n(x) \cdot v|}{(2T_w(x))^2}
\exp\left(- \frac{1}{2T_w(x)}\left[
\frac{|v_\perp|^2 + (1- r_\perp) |u_\perp|^2}{r_\perp}
+ \frac{|v_\parallel - (1- r_\parallel ) u_\parallel|^2}{r_\parallel (2- r_\parallel)}
\right]\right)\\
& \times  I_0 \left(
 \frac{1}{2T_w(x)}\frac{2 (1-r_\perp)^{1/2} v_\perp u_\perp}{r_\perp}
\right).
\end{split}
\end{equation}
Here $T_w(x)$ is a wall temperature on the boundary and $I_0 (y) := \pi^{-1} \int^{\pi}_0e^{y \cos \phi } d \phi$.
In this formula, $v_\perp$ and $v_\parallel$ denote the normal and tangential components of the velocity respectively: $   v_\perp= v\cdot n(x) ,\quad v_\parallel = v- v_\perp n(x)\,$.
%
Similarly $u_\perp= u\cdot n(x)$ and $u_\parallel = u- u_\perp n(x)$.

This model can be considered as a generalization of fundamental boundary conditions. For instance by setting $r_\perp=1$ and $r_\parallel=1$, the scattering kernel equals $R(u\to v;x,t)=\frac{2}{\pi (2T_w(x))^2}e^{-\frac{|v|^2}{2T_w(x)}} |n(x)\cdot v|.$
%
This corresponds to the diffuse boundary condition \eqref{diffuse_BC}. With $r_\perp=0,r_\parallel=0$, the scattering kernel is given by $R(u\to v;x,t)=\delta(u-\mathfrak{R}_xv)$,
  with $\mathfrak{R}_xv=v-2n(x)(n(x)\cdot v)$. This corresponds the specular reflection boundary condition $F(t,x,v)=F(t,x,\mathfrak{R}_xv)$. With $r_\perp=0,r_\parallel=2$, the scattering kernel is given by $R(u\to v;x,t)=\delta(u+v)$, which corresponds the bounce-back reflection reflection boundary condition $F(t,x,v)=F(t,x,-v)$.
   
%
%
It is important to note that the C-L boundary condition satisfies the reciprocity property
\Be\label{eqn: reciprocity}
    R(u\to v;x,t)=R(-v\to -u;x,t) \frac{e^{-|v|^2/(2T_w(x))}}{e^{-|u|^2/(2T_w(x))}}\frac{|n(x)\cdot v|}{|n(x)\cdot u|}\,,
\Ee
and the normalization property $\int_{\gamma_-(x)} R(u\to v;x,t) dv=1$.

Define the global Maxwellian using the maximum wall temperature as
\begin{equation}\label{eqn: def for weight}
\mu:=e^{-\frac{|v|^2}{2T_M}}\,,\ \text{ with }T_M:=\max_{x\in \partial \Omega}\{T_w(x)\}. \  \text{ And set } F = \sqrt \mu f.
\end{equation}
The main result in \cite{CKL} is:
\begin{theorem*}\label{local_existence}
Assume $\Omega \subset \mathbb{R}^3$ is open bounded, and convex $C^3$ domain. A wall temperature $T_w(x)>0$ is defined on $x\in \partial \Omega$ and smooth. Assume that two accommodation coefficients satisfy

\begin{equation}\label{eqn: Constrain on T}
\frac{\min_{x\in \partial \Omega}\{T_w(x)\}}{\max_{x\in \partial \Omega}\{T_w(x)\}}>\max\Big(\frac{1-r_\parallel}{2-r_\parallel},\frac{\sqrt{1-r_\perp}-(1-r_\perp)}{r_\perp}\Big)\,.
\end{equation}
Let $0< \tilde{\theta}< \theta <\frac{1}{4\max_{x\in \partial \Omega}\{T_w(x)\}}.$
%
Assume
\begin{equation}
\| w_\theta f_0 \|_\infty < \infty, \, \label{eqn: w f_0} \| w_{\tilde{\theta}} \nabla_v f_0 \|_{L^{3}_{x,v}}<\infty,
\end{equation}
\begin{equation}
\| w_{\tilde{\theta}} \alpha_{f_0, \epsilon }^\beta \nabla_{x,v } f_0 \|_{ {L}^{p } ( {\O} \times \R^3)} <\infty\quad \text{for} \quad 3< p < 6\,,\, 1-\frac{2}{p }< \beta< \frac{2}{3}.\label{eqn: alpha f0}
\end{equation}
Then there is a unique solution $F(t,x,v) = \sqrt \mu f(t,x,v)$ to~\eqref{Boltzmann_E}, \eqref{eqn:BC}, \eqref{eqn: Formula for R} in a time interval of $t \in [0,\bar{t}]$.
%
Moreover, there are $\mathfrak{C}>0$ and $\lambda>0$, so that $f$ satisfies
\begin{equation}\label{infty_local_bound}
\sup_{0 \leq t \leq \bar{t}}\| w_{\theta}e^{-\mathfrak{C} \langle v\rangle^2 t} f  (t) \|_{\infty}\lesssim \| w_\theta f_0 \|_\infty  , \ \ \sup_{0 \leq t \leq \bar{t}}\| \nabla_v f (t) \|_{L^3_xL^{1+ \delta}_v}< \infty,
\end{equation}
\begin{equation}\label{W1p_local_bound}
\sup_{0 \leq t \leq \bar{t}}\Big\{ \| w_{\tilde{\theta}}e^{-\lambda t\langle v\rangle }\alpha_{f,\epsilon }^\beta \nabla_{x,v} f (t) \|_{p} ^p+ \int^t_0  |w_{\tilde{\theta}} e^{-\lambda s\langle v\rangle } \alpha_{f,\epsilon}^\beta \nabla_{x,v} f (t) |_{p,+}^p\Big\}< \infty .
\end{equation}
\end{theorem*}

\unhide

\section{On the Vlasov-Poisson-Boltzmann system surrounded by Conductor boundary}

In the second part of the paper, we consider the one-species VPB system surrounded by conductor boundary. More specifically, we consider the system \eqref{Boltzmann_E}, \eqref{Field}, where the electrostatic potential $\phi$ is obtained by
\Be \label{phiC}
-\Delta_x \phi(t,x) = \int_{\mathbb R^3} F(t,x,v) dv, \, x \in \O, \ \  \ 
\phi = 0, \, x \in \p \O.
\Ee
An important benefit in the conductor boundary setting \eqref{phiC} is that $E = - \nabla_x \phi$ enjoys the sign condition \eqref{signEonbdry} from a quantitative Hopf lemma, without the need of an external potential.
\begin{lemma}[Lemma $3.2$ in \cite{BC}]\label{Hopf}
Suppose $ h \ge 0 $, and $h \in L^\infty(\O)$. Let $v$ be the solution of
\Be  \begin{split}
- \Delta v = h   \text{ in } \O,
 \ \ v = 0   \text{ on } \p \O.
\end{split} \Ee
Then for any $ x \in \p \O$,
\Be \label{hopfquan}
\frac{ \p v(x) }{ \p n} \ge c \int_{\O} h(x)  d (x, \p \O)dx,
\Ee
for some $c > 0$ depending only on $\O$. Here $d (x, \p \O)$ is the distance from $x$ to the boundary $\p \O$.
\end{lemma}
\hide
\begin{proof}
See .

\end{proof}\unhide

Our goal is to prove a local existence and regularity theorem for the system \eqref{Boltzmann_E}, \eqref{Field}, \eqref{diffuse_BC}, \eqref{phiC}. Let's first define our distance function $\tilde \alpha$.

%

Let $d(x,\partial \Omega) := \inf_{y \in \partial \Omega} \| x - y \| $. For any $\delta > 0$, let $ \Omega ^\delta : = \{ x \in \Omega : d(x, \partial \Omega ) < \delta \}$. For $\delta \ll 1$ is small enough, we have for any $x \in \Omega ^\delta$ there exists a unique $\bar x \in \partial \Omega$ such that $d(x,\bar x ) = d(x, \partial \Omega)$ (cf. (2.44) in \cite{VPBEP}).

\begin{definition}
First we define for all $(x, v ) \in  \Omega ^{\delta} \times \mathbb R^3 $,
\[
\beta(t,x,v) = \bigg[ |v \cdot \nabla \xi (x)| ^2 + \xi (x)^2 - 2 (v \cdot \nabla^2 \xi(x) \cdot v ) \xi(x) +2(\nabla \phi(t,\overline x ) \cdot \nabla \xi (\overline x ) )\xi(x) \bigg]^{1/2}.
\]
For any $\epsilon >0$, let $\chi_\epsilon: [0,\infty) \to [0,\infty) $ be a smooth function satisfying $  \chi_\epsilon(x) =x$ for $0 \le x \le \frac{\epsilon}{4}$, $\chi_\epsilon(x) = C_\epsilon$ for  $x \ge \frac{\epsilon}{2}$,  $\chi_\epsilon(x)$ is increasing for $ \frac{\epsilon}{4} < x < \frac{\epsilon}{2}$, and $ \chi_\epsilon'(x) \le 1$. Let $\delta': = \min \{| \xi (x)| : x\in \Omega, d(x, \partial \Omega) = \delta \} $, then we define our weight function to be:

\begin{equation} \label{alphadef}
\tilde \alpha(t,x,v) : = 
\begin{cases}
 (\chi_{\delta'}  ( \beta(t,x,v) ) )^{} & x \in  \Omega^\delta, \\
 C_{\delta'} ^{} & x \in \Omega \setminus  \Omega^\delta.
\end{cases} \end{equation}
\end{definition}

\begin{theorem}[Weighted $W^{1,\infty}$ estimate for the VPB surrounded by conductor]  \label{WlinftyVPBthm} 
Assume $F_0 = \sqrt \mu f_0$ satisfies
\Be \label{VPBf0assumption}
\|  w_\vartheta \tilde \alpha \nabla_{x,v} f_0 \|_{L^\infty(\bar \O \times \mathbb R^3)} + \| w_\vartheta f_0 \|_{L^\infty(\bar \O \times \mathbb R^3)} + \| w_\vartheta \nabla_v f_0 \|_{L^3(\bar \O \times \mathbb R^3)} < \infty,
\Ee
for some $0< \vartheta < \frac{1}{4}$.Then there exists a unique solution $F(t,x,v) = \sqrt \mu f(t,x,v) $ to (\ref{Boltzmann_E}), \eqref{Field}, \eqref{diffuse_BC}, (\ref{phiC}) for $t \in [0,T]$ with $0 < T \ll 1$, such that for some $0< \vartheta' < \vartheta $, $\varpi \gg 1$,
\begin{equation}  \label{linfinitybddsolution}
\sup_{0\le t \le T} \| w_{\vartheta'} f(t) \|_{L^\infty(\bar \O \times \mathbb R^3)} < \infty,
\end{equation} 
\begin{equation}\label{weightedC1bddsolution}
\sup_{0 \le t \le T}  \| w_{\vartheta'} e^{-\varpi  \langle v \rangle t } \tilde \alpha  \nabla_{x,v}  f^{} (t,x,v) \|_{L^\infty(\bar \O \times \mathbb R^3)} < \infty, 
\end{equation}
\begin{equation} \label{L3L1plusbddsolution}
 \sup_{0 \le t \le T}\| e^{-\varpi \langle v \rangle t } \nabla_v f(t) \|_{L^3_x(\Omega)L_v^{1+\delta}(\mathbb R^3 ) } < \infty \text{ for } 0< \delta \ll 1.
\end{equation}
\end{theorem}



The corresponding equation for $f = \frac{F}{\sqrt \mu } $ is
\begin{equation} \label{VPBsq1}
(\partial_t + v \cdot \nabla_x - \nabla \phi \cdot \nabla_v  + \frac{v}{2} \cdot \nabla \phi + \nu( \sqrt \mu f) ) f^{} = \Gamma_{\text{gain}} (f,f),
\end{equation}
\begin{equation} \label{VPBsq3}
-\Delta_x \phi(t,x) = \int_{\mathbb R^3 } \sqrt \mu f dv, \,\,   \phi = 0 \text{  on } \partial \Omega,
\end{equation}
\Be \label{fbdry}
f^{}(t,x,v) = c_\mu \sqrt{\mu(v)} \int_{n\cdot u >0 } f(t,x,v) \sqrt{\mu(u)} (n(x) \cdot u ) du.
\Ee
Here $
 \nu (\sqrt \mu f )(v) : = \int_{\mathbb R^3 } \int_{\mathbb S^2} |v - u|^{\kappa} q_0 (\frac{ v -u}{ |v -u| } \cdot w ) \sqrt{\mu(u) } f(u) d\omega du$, and $
 \Gamma_{\text{gain}} (f_1,f_2) (v): = \int_{\mathbb R^3 } \int_{\mathbb S^2} |v - u|^{\kappa} q_0 (\frac{ v -u}{ |v -u| } \cdot w ) \sqrt{\mu(u) } f_1(u')f_2(v') d\omega du$.

Let $\partial \in \{ \nabla_x, \nabla_v \}$. Let $E = - \nabla_x \phi $. Denote 
\Be \label{nuvarpi}
\nu_{\varpi} = \nu(\sqrt \mu f )+ \frac{v}{2} \cdot E  + \varpi \langle v \rangle  + t \varpi \frac{v}{\langle v \rangle }\cdot E    -  {\tilde \alpha}^{-1} ( \partial_t {\tilde \alpha} + v \cdot \nabla_x {\tilde \alpha} + E \cdot \nabla_v {\tilde \alpha} ).
\Ee
Then by direct computation we get 
 \begin{equation} \label{seqforc1fixedpotential} \begin{split}
 \bigg \{ & \partial_t + v\cdot \nabla_x  + E \cdot \nabla_v + \nu_{\varpi} \bigg\} ( e^{-\varpi \langle v \rangle t } {\tilde \alpha} \partial f^{})
\\  = & e^{-\varpi \langle v \rangle t }  {\tilde \alpha} \left(  \partial \Gamma_{gain} (f,f) - \partial v \cdot \nabla_x f^{} - \partial E \cdot \nabla_v f^{} - \partial ( \frac{v}{2} \cdot E ) f^{} - \partial (\nu (\sqrt \mu f ) ) f^{}  \right)
\\  := & \mathcal{N}(t,x,v).
\end{split} \end{equation}

In order to deal with the diffuse boundary condition \eqref{diffuse_BC}, we define the stochastic (diffuse) cycles as $(t^0,x^0,v^0) = (t,x,v)$,
\begin{equation} \label{diffusecycles} \begin{split}
& t^1 = t - \tb(t,x,v), \, x^1 = \xb(t,x,v) = X(t - \tb(t,x,v);t,x,v), 
\\ & v_b^0 = V(t - \tb(t,x,v);t,x,v) = \vb(t,x,v),
\end{split} \end{equation}
 and $v^1 \in \mathbb R^3$ with $n(x^1) \cdot v^1 > 0$. 
For $l\ge1$, define
\[ \begin{split}
& t^{l+1} = t^l - \tb(t^l,x^l,v^l), x^{l + 1 } = \xb(t^l,x^l,v^l), 
\\ & v_b^l = \vb(t^l,x^l,v^l), 
\end{split} \]
and $v^{l+1} \in \mathbb R^3 \text{ with } n(x^{l+1}) \cdot v^{l+1} > 0$.
Also, define 
\[
X^l(s) = X(s;t^l,x^l,v^l), \, V^l(s) = V(s;t^l,x^l,v^l),
\]
 so $X(s) = X^0(s), V(s) = V^0(s)$. We have the following lemma.

\begin{lemma}[Lemma $12$ in \cite{VPBEP}] \label{expandtrajl}


If $t^1 < 0$, then
\Be \label{C1trajectoryt1less0}
\begin{split}
&  e^{-\varpi \langle v \rangle t } \tilde \alpha | \partial  f (t,x,v) |
\\ &  \lesssim \tilde \alpha(0,X^0(0), V^0(0) ) \partial f (0, X^0(0) , V^0(0) ) + \int_0^t \mathcal N(s,X^0(s), V^0(s) ) ds.
\end{split} \Ee

If $t^1 > 0$, then
\begin{equation} \label{C1trajectory}
\begin{split}
& e^{-\varpi \langle v \rangle t } \tilde \alpha | \partial  f (t,x,v) |
\\ \lesssim & e^{-\frac{\vartheta}{2}  |\vb^0| ^2} P(\| w_\vartheta f_0 \|_\infty) +  \int_{t^1}^t \mathcal N(s, X^0(s), V^0(s) ) ds
\\ & +  \sqrt {\mu (\vb^0) } \langle \vb^0 \rangle^2 \int_{\prod_{j=1}^{l-1} \mathcal V_j} \sum_{i=1}^{l-1} \textbf{1}_{\{t^{i+1} < 0 < t^i \}}  |\tilde \alpha  \partial f  (0,X^{i}(0), V^{i}(0)) |    \, d \Sigma_{i}^{l-1}
\\ & +  \sqrt {\mu (\vb^0) } \langle \vb^0 \rangle^2 \int_{\prod_{j=1}^{l-1} \mathcal V_j} \sum_{i=1}^{l-1} \textbf{1}_{\{t^{i+1} < 0 < t^i \}}  \int_0^{t^i} \mathcal N(s, X^i(s), V^i(s) ) ds  \, d \Sigma_{i}^{l-1}
\\ & +  \sqrt {\mu (\vb^0) } \langle \vb^0 \rangle^2 \int_{\prod_{j=1}^{l-1} \mathcal V_j} \sum_{i=1}^{l-1} \textbf{1}_{\{t^{i+1} > 0 \}}  \int_{t^{i+1}}^{t^i} \mathcal N(s, X^i(s), V^i(s) ) ds  \, d \Sigma_{i}^{l-1}
\\ & +  \sqrt {\mu (\vb^0) } \langle \vb^0 \rangle^2 \int_{\prod_{j=1}^{l-1} \mathcal V_j} \sum_{i=2}^{l-1} \textbf{1}_{\{t^{i} > 0 \}}   e^{-\frac{\vartheta}{2}  |\vb^{i-1}| ^2} P(\| w_\vartheta f_0 \|_\infty)  \, d \Sigma_{i-1}^{l-1}
\\ & +  \sqrt {\mu (\vb^0) } \langle \vb^0 \rangle^2  \int_{\prod_{j=1}^{l-1} \mathcal V_j}  \textbf{1}_{\{t^{l} > 0 \}}  e^{-\varpi \langle \vb^{l-1} \rangle t^l } \tilde \alpha (t^l,x^l, \vb^{l-1}) | \partial  f (t^l,x^l,\vb^{l-1}) | d \Sigma_{l-1}^{l-1},
\end{split} \end{equation}
where $\mathcal V_j = \{ v^j \in \mathbb R^3: n(x^j ) \cdot v^j > 0 \}$,
and 
\[ \begin{split}
d \Sigma_i^{l-1} = & \{\prod_{j=i+1}^{l-1}  \mu(v^j) c_\mu |n(x^j ) \cdot v^j | dv^j \} 
\{ e^{\varpi \langle v^i \rangle t^i } \mu^{1/4}(v^i) \langle v^i \rangle d v^i \}
\\ & \{\prod_{j=1}^{i-1} \sqrt{\mu(\vb^j ) } \langle \vb^j \rangle \mu^{1/4}(v^j ) \langle v^j \rangle e^{\varpi \langle v^j \rangle t^j } d v^j\},
\end{split} \]
where $c_\mu$ is the constant that $ \int_{\mathbb R^3 }  \mu(v^j) c_\mu |n(x^j ) \cdot v^j | dv^j = 1$.
\end{lemma} 

The following lemma is necessary for us to establish Theorem \ref{WlinftyVPBthm}.

\begin{lemma}
If $(F, \phi)$ solves (\ref{phiC}), write $f = \frac{F}{\sqrt \mu}$, then
\begin{equation} \label{linfinitybdofpotential}
\| \phi_F(t) \|_{C^{1,1-\delta} (\O) } \lesssim_{\delta,\Omega} \| w_\vartheta f(t) \|_{L^\infty(\bar \O \times \mathbb R^3)}, \text{ for any } 0< \delta <1,
\end{equation}
and
\begin{equation} \label{c2bdofpotential}
\| \nabla ^2 \phi_F(t) \|_{L^\infty(\O)} \lesssim \| w_\vartheta f(t) \|_{L^\infty(\bar \O \times \mathbb R^3)} + \| e^{-\varpi \langle v \rangle t } \tilde \alpha \nabla_{x} f (t) \|_{L^\infty(\bar \O \times \mathbb R^3)}.
\end{equation}
\end{lemma}

\begin{proof}
It is obvious to have (\ref{linfinitybdofpotential}) from the Morrey inequality and elliptic estimate.
Next we show (\ref{c2bdofpotential}).
By Schauder estimate, we have, for $p >3$ and $\Omega \subset \mathbb R^3$,
\[
\| \nabla^2 \phi_F(t) \|_{L^\infty(\O)} \le \| \phi_F \|_{C^{2, 1 - \frac{3}{p}}(\O)} \lesssim_{p,\Omega} \| \int_{\mathbb R^3} f (t) \sqrt \mu dv \|_{C^{0, 1 - \frac{3}{p}}(\O)}.
\]
Then by Morrey inequality, $W^{1,p} \subset C^{0, 1 -\frac{3}{p}} $ with $p >3$ for a domain $\Omega \subset \mathbb R^3$ with a smooth boundary $\partial \Omega$, we derive
\[ \begin{split}
 \| & \int_{\mathbb R^3} f (t) \sqrt \mu dv \|_{C^{0, 1 - \frac{3}{p}}}   \lesssim  \| \int_{\mathbb R^3} f (t) \sqrt \mu dv \|_{W^{1,p} }
 \\ 
  & \lesssim  \| w_\vartheta f(t) \|_\infty +  \| e^{-\varpi \langle v \rangle t } \tilde \alpha \nabla_{x} f (t) \|_\infty \| \int_{\mathbb R^3 } e^{ \varpi \langle v \rangle t } \sqrt \mu \frac{1}{\tilde \alpha} dv \|_{L^p(\Omega ) }.
\end{split} \]


It suffices to show that for some $\beta > 1$, 
\Be \label{int1alphaLpbdd}
\| \int_{\mathbb R^3 } e^{-\frac{1}{8} |v|^2 }  \frac{1}{\tilde \alpha^\beta} dv \|_{L^p(\Omega ) } < \infty.
\Ee
 Since $\tilde \alpha$ is bounded from below when $x$ is away from the boundary of $\Omega$, it suffices to only consider the case when $x$ is close enough to $\partial \Omega$. From direct computation (see \cite{VPBEP}), we get
\begin{equation} \label{int1overalphadv}
 \int_{\mathbb R^3 } e^{-\frac{1}{8} |v|^2 }  \frac{1}{\tilde \alpha^\beta} dv \lesssim \frac{1}{(\xi(x)^2 - 2 E(t,\bar x ) \cdot \nabla \xi(\bar x) \xi(x) )^{\frac{\beta -1 }{2}}} \lesssim \frac{1}{ |\xi(x) | ^{\frac{\beta -1 }{2}}}.
\end{equation}
%
And since $\xi$ is $C^2$, we have 
\[
\int_{d(x,\partial \Omega) \ll 1 } \frac{1}{ |\xi(x) | ^{\frac{(\beta -1 )p}{2}}} dx \lesssim \int_{d(x,\partial \Omega) \ll 1 } \frac{1}{ |x - \bar x | ^{\frac{(\beta -1 )p}{2}}} dx.
\]
Now from (\ref{convexity_eta}), 
\[
\int_{\Omega \cap B(p;\delta_2 ) }  \frac{1}{|x - \bar x  | ^{\frac{(\beta -1 )p}{2}}} dx \lesssim \int_{ |x_n | < \delta_1 } \frac{1}{|x_n|^{\frac{(\beta -1 )p}{2}}} d_{x_n} < \infty,
\]
if we pick $\beta < \frac{2}{p}+1 $. And since $\partial \Omega$ is compact, we can cover $\partial \Omega$ with finitely many such balls, and therefore we get (\ref{int1alphaLpbdd}).

\end{proof}

\begin{proof} [Proof of Theorem \ref{WlinftyVPBthm}]
For the sake of simplicity we only show the a priori estimate. See \cite{CKL} for the construction of the sequences of solutions and passing a limit.


The proof of \eqref{linfinitybddsolution} for $f$ satisfying \eqref{VPBsq1}, \eqref{VPBsq3}, and \eqref{fbdry} is standard. We refer to Theorem $4$ in \cite{VPBEP}.



First from \eqref{VPBsq3} and the fact that $ \int_{\mathbb R^3 } \sqrt \mu f dv \ge 0$, we apply Lemma \ref{Hopf} to get
\Be \label{hopfm}
 - \frac{  \p \phi (t,x)}{\p n} \ge c \iint_{\O \times \mathbb R^3} \sqrt \mu  f(t,x,v)  \delta(x) dv dx,
\Ee
for some $c$ depending only on $\O$. 

Denote
\[
 \iint_{\O \times \mathbb R^3}  F_0(x,v)  \delta(x) dv dx = c_{E_0}.
\]
Then $\int_0^T \iint_{\O \times \mathbb R^3 }\delta(x) \times \eqref{Boltzmann_E} \ dv dx dt$ gives
\[ \begin{split}
&  \iint_{\O \times \mathbb R^3}F(T,x,v)  \delta(x) dv dx  
 \\ &  =   \iint_{\O \times \mathbb R^3}  F_0(x,v)  \delta(x) dv dx + \int_0^T \iint_{\O \times \mathbb R^3 } F v \cdot \nabla_x \delta(x)  dv dx dt.
\end{split} \]
Together with \eqref{linfinitybddsolution} and \eqref{hopfm} we deduce
\Be \label{signphim}
 - \frac{  \p \phi (t,x)}{\p n} \ge c \iint_{\O \times \mathbb R^3}F(t,x,v)  \delta(x) dv dx >  \frac{c c_{E_0}}{2},
\Ee
as long as $ T \lesssim \frac{ c_{E_0}}{2 M}$.

Next, we investigate \eqref{seqforc1fixedpotential}. Since
\[ \begin{split}
 w_{\vartheta } \Gamma_{gain}(\partial f,f )  
 \lesssim \| e^{2\vartheta' |v|^2 } f \|_\infty \int_{\mathbb R^3 } \frac{ e^{-C_{\vartheta'} |u -v | ^2 } } {|u -v |^{2 -\kappa} }  |e^{\vartheta ' |u|^2 } \partial f(t,x,u ) | du,
\end{split} \]
and
\[ \begin{split}
w_{\vartheta } \nu(\sqrt{\mu} \partial f ) f 
\lesssim  \| e^{2 \vartheta' |v|^2 } f \|_\infty  \int_{\mathbb R^3 } \frac{ e^{-C_{\vartheta'} |u -v | ^2 } } {|u -v |^{2 -\kappa} } | \partial f(t,x,u ) | du.
\end{split} \]
Thus from (\ref{linfinitybddsolution}) we have the following bound for $\mathcal N$:

\begin{equation} \label{Nbdd}\begin{split}
|  \mathcal   N(t,x,v) | 
  \lesssim &   (1 + \| \nabla^2 \phi \|_\infty)  [ P(\| w_\vartheta f_0 \|_\infty )  + |w_{\vartheta'}  e^{-\varpi \langle v \rangle t  }  \tilde \alpha  \partial f(t,x,v) | ]
\\ & +  \| w_\vartheta f_0 \|_\infty  e^{-\varpi \langle v \rangle t }  \tilde \alpha(t,x,v) \int_{\mathbb R^3 }  \frac{ e^{-C_\vartheta |u -v | ^2 } } {|u -v |^{2 -\kappa} } | e^{\vartheta ' |u|^2 }\partial f(t,x,u ) | du.
\end{split} \end{equation}

Recall the definition of $\nu_\varpi$ in \eqref{nuvarpi}, note that from the velocity lemma (\ref{vlemma}), and (\ref{signphim}) we have
\[ \begin{split}
& \tilde \alpha^{-1}  ( \partial_t \tilde \alpha + v \cdot \nabla_x \tilde \alpha -\nabla \phi \cdot \nabla_v \tilde \alpha ) 
 \\ \lesssim & ( \| \nabla \phi \|_\infty + \| \nabla^2 \phi \|_\infty ) \langle v \rangle
 \\ \lesssim  & (\| w_{\vartheta'} f(t) \|_\infty + \| e^{-\varpi \langle v \rangle t } \tilde \alpha \nabla_{x} f (t) \|_\infty)\langle v \rangle
\\  \lesssim & (P(\| w_{\vartheta} f_0\|_\infty ) + \| \tilde \alpha \partial f_0 \|_\infty)\langle v \rangle.
   \end{split} \]
Therefore we have
\Be \label{largeomegabar}
\nu_\varpi \ge  \frac{\varpi}{2} \langle v \rangle,
\Ee
once we choose $\varpi \gg 1$ large enough.
   
For $t^1 < 0$, using the Duhamel's formulation we have from (\ref{seqforc1fixedpotential})

\begin{equation}  \label{}
\begin{split}
 w_{\vartheta'} & e ^{-\varpi \langle v \rangle t  }  \tilde \alpha | \partial  f (t,x,v) |
\\   \le &  e^{ -\int_s^t \nu_{\varpi} (\tau, X(\tau), V(\tau) d\tau}  e^{\vartheta' |V(0)|^2 } \tilde \alpha \partial f (0, X(0) , V(0) ) 
\\ & + \int_0^t e^{ -\int_s^t \nu_{\varpi} (\tau, X(\tau), V(\tau) d\tau } \mathcal N(s,X(s), V(s) ) ds.
\end{split} 
\end{equation}
Thus by (\ref{Nbdd}) we have
\[ \begin{split}
 \sup_{0 \le t \le T} & \| \textbf{1}_{ \{ t^1 < 0 \}} e^{-\varpi \langle v \rangle t } w_{\vartheta'} \tilde \alpha  \partial  f (t,x,v) \|_\infty
\\ \le & \sup_{0 \le t \le T}  \| e^{ -\int_0^t \nu_{\varpi} (\tau, X(\tau), V(\tau) d\tau} e^{\vartheta' |V(0)|^2 } \tilde \alpha \partial f (0, X(0) , V(0) ) 
\\ & + \int_0^t e^{ -\int_s^t \nu_{\varpi} (\tau, X(\tau), V(\tau) d\tau } \mathcal N(s,X(s), V(s) ) ds \|_\infty
 \\ \le & \| w_{\vartheta'} \tilde \alpha \partial f_0 \|_\infty + P(\| w_{\vartheta} f_0 \|_\infty)   \sup_{0 \le t \le T}  \| w_{\vartheta'} e^{-\varpi \langle v \rangle t  } \tilde \alpha  \partial  f (t,x,v) \|_\infty
 \\ & +T   (1 + \| \nabla^2 \phi \|_\infty) [ P(\| w_{\vartheta} f_0 \|_\infty) +   \sup_{0 \le t \le T}  \| w_{\vartheta'}  e^{-\varpi \langle v \rangle t  } \tilde \alpha  \partial  f (t,x,v) \|_\infty] 
 \\ & \times \int_0^t \int_{\mathbb R^3} e^{ -\int_s^t\frac{\varpi}{2} \langle V(\tau;t,x,v) \rangle d\tau }  \frac{e^{-\varpi  \langle (s;t,x,v) \rangle  s }}{e^{-\varpi  \langle u \rangle s}} \frac{e^{-C_\vartheta |V(s)-u|^2 }}{|V(s) -u |^{2 - \kappa}  }  \frac{\tilde \alpha(s,X(s) ,V(s))}{\tilde \alpha(s,X(s) ,u)} du  ds.
\end{split} \]
Now since
$
 \langle u \rangle -  \langle V(s;t,x,v) \rangle \le 2 \langle u - V(s;t,x,v) \rangle,
$
we have
$
 \frac{e^{-\varpi  \langle (s;t,x,v) \rangle  s }}{e^{-\varpi  \langle u \rangle s}} e^{-C_\vartheta |V(s)-u|^2 }
\lesssim   e^{ -\frac{ C_\vartheta |V(s) - u | ^2}{2}}.
$
Thus
\begin{equation} \label{genlemma1toget}
\begin{split}
\int_0^t & \int_{\mathbb R^3} e^{ -\int_s^t\frac{\varpi}{2} \langle V(\tau;t,x,v) \rangle d\tau }  \frac{e^{-\varpi  \langle (s;t,x,v) \rangle  s }}{e^{-\varpi  \langle u \rangle s}}  \frac{e^{-C_\vartheta |V(s)-u|^2 }}{|V(s) -u |^{2 - \kappa}  }  \frac{\tilde \alpha(s,X(s) ,V(s))}{\tilde \alpha(s,X(s) ,u)} du  ds
\\ \lesssim &    \int_0^t \int_{\mathbb R^3} e^{ -\int_s^t\frac{\varpi}{2} \langle V(\tau;t,x,v) \rangle d\tau }  \frac{e^{-\frac{C_\vartheta}{2} |V(s)-u|^2 }}{|V(s) -u |^{2 - \kappa}  }  \frac{\tilde \alpha(s,X(s) ,V(s))}{\tilde \alpha(s,X(s) ,u)} du  ds.
\end{split} 
\end{equation}
Note that, for any $\beta > 1$,
$
\frac{1}{\tilde \alpha(x, X(s) ,u ) } \lesssim \frac{1}{ (\tilde \alpha (x, X(s) , u ) )^\beta } + 1.
$
So from (\ref{signphim}) we can let $1 < \beta \le 2$, and apply the nonlocal-to-local estimate (\ref{nonlocaltolocal}) to (\ref{genlemma1toget}) to have
\begin{equation}
\begin{split}
 \int_0^t  & \int_{\mathbb R^3} e^{ -\int_s^t\frac{\varpi}{2} \langle V(\tau;t,x,v) \rangle d\tau }  \frac{e^{-\varpi  \langle (s;t,x,v) \rangle  s }}{e^{-\varpi  \langle u \rangle s}} \frac{e^{-C_\vartheta |V(s)-u|^2 }}{|V(s) -u |^{2 - \kappa}  }  \frac{\tilde \alpha(s,X(s) ,V(s))}{\tilde \alpha(s,X(s) ,u)} du  ds
\\ \lesssim &  e^{ C ( \| \nabla \phi \|_\infty^2 + \| \nabla^2 \phi \|_\infty )}   \left( \frac{ \delta^{\frac{3 - \beta}{2}} (\tilde \alpha(t,x,v) )^{3 - \beta } }{ (|v|^2 + 1 )^{\frac{3 -\beta}{2}} } + \frac{ (|v| + 1 )^{\beta - 1} (\tilde \alpha(t,x,v) )^{2 - \beta }}{ \delta ^{\beta - 1} \varpi \langle v \rangle  }   \right)
\\ \lesssim &  e^{ C ( \| \nabla \phi \|_\infty^2 + \| \nabla^2 \phi \|_\infty  ) }  \left( \delta^{\frac{3 - \beta}{2}} + \frac{1}{\delta ^{\beta - 1} \varpi} \right),
\end{split}
\end{equation}
where we used $\tilde \alpha(s,X(s) ,V(s)) \lesssim e^{ C ( \| \nabla \phi \|_\infty^2 + \| \nabla^2 \phi \|_\infty ) }  \tilde \alpha(t,x,v)$.

Similarly, for $t^1(t,x,v) \ge 0 $, we again apply the nonlocal-to-local estimate (\ref{nonlocaltolocal}) to get
\[
\begin{split}
| &  \textbf{1}_{ \{ t^1 > 0 \}}  w_{\vartheta'} e^{-\varpi \langle v \rangle t } \tilde \alpha  \partial  f (t,x,v) |
\\ \lesssim & C_l e^{Cl t^2 } \left( \delta^{\frac{3 - \beta}{2}} + \frac{1}{\delta ^{\beta - 1} \varpi}  \right)  P(\| w_{\vartheta} f_0 \|_\infty) \max_{ 0 \le i \le l-1 } e^{ C ( \| \nabla \phi \|_\infty^2 + \| \nabla^2 \phi \|_\infty  )}   
\\ & \times  \sup_{0 \le t \le T}  \| w_{\vartheta'} e^{-\varpi \langle v \rangle t } \tilde \alpha  \partial  f (t,x,v) \|_\infty
\\ & +  T(1 + \|\nabla^2 \phi \|_\infty) \sup_{0 \le t \le T}  \| w_{\vartheta'} e^{-\varpi  \langle v \rangle t } \tilde \alpha  \partial  f (t,x,v) \|_\infty 
\\  + &  T  l (Ce^{Ct^2 } )^l  (1 + \|\nabla^2 \phi \|_\infty)  \sup_{0 \le t \le T} \| w_{\vartheta'} e^{-\varpi \langle v \rangle t } \tilde \alpha  \partial  f (t,x,v) \|_\infty  
\\ + & Tl (Ce^{Ct^2 } )^l    (1 + \|\nabla^2 \phi  \|_\infty )P(\| w_{\vartheta} f_0 \|_\infty) 
+  l (Ce^{Ct^2 } )^l  \| \tilde \alpha \partial f_0  \|_\infty +   P(\| w_{\vartheta} f_0 \|_\infty) 
\\ & + C \left( \frac{1}{2} \right) ^l   \sup_{0 \le t \le T} \|w_{\vartheta'} e^{-\varpi \langle v \rangle t} \tilde \alpha  \partial  f (t,x,v) \|_\infty.
\end{split}
\]

Finally from \eqref{c2bdofpotential}, we can choose a large $l$ then large $C$ then small $\delta$ then large $\varpi$ and finally small $T$ to conclude 
\[ \begin{split}
 \sup_{0 \le t \le T}   \| e^{-\varpi \langle v \rangle t } \tilde \alpha  \partial  f (t,x,v) \|_\infty
\le  \frac{C_1}{2} \left(  \| w_{\vartheta } \tilde \alpha \partial f_0 \|_\infty + P ( \| w_\vartheta f_0 \|_\infty ) \right)
\end{split} \]
This proves (\ref{weightedC1bddsolution}).

Next we prove \eqref{L3L1plusbddsolution}. Consider taking $\nabla_v$ derivative of (\ref{VPBsq1}) and adding the weight function $e^{-\varpi \langle v \rangle t}$, we get
\Be\begin{split}\label{eqtn_g_v}
			&[\p_t + v\cdot \nabla_x - \nabla_x \phi \cdot \nabla_v + \frac{v}{2} \cdot \nabla_x \phi + \varpi \langle v \rangle - \frac{v}{ \langle v \rangle } \varpi t \cdot \nabla_x \phi + \nu(\sqrt \mu f ) ]  (e^{-\varpi \langle v \rangle t} \nabla_v f )
			\\ =& e^{-\varpi \langle v \rangle t} \left( - \nabla_v \nu(\sqrt \mu f ) f - \nabla_x f - \frac{1}{2} \nabla_x\phi f + \nabla_v \Gamma_{\text{gain}}(f,f) \right),
		\end{split}\Ee
with the boundary bound for $(x,v) \in\gamma_-$
		\Be\label{bdry_g_v}
		\big|\nabla_v f  \big| \lesssim   |v| \sqrt{\mu} \int_{n \cdot u>0} |f| \sqrt{\mu} \{n \cdot u \} \dd u \ \ \text{on } \ \gamma_-.
		\Ee
And 
\[
 \frac{v}{2} \cdot \nabla_x \phi + \varpi \langle v \rangle - \frac{v}{ \langle v \rangle } \varpi t \cdot \nabla_x \phi + \nu(\sqrt \mu f ) > \frac{\varpi }{2} \langle v \rangle,
 \]
 for $\varpi \gg 1$.

		Using the Duhamel's formulation, from (\ref{eqtn_g_v}) we obtain the following bound along the characteristics
		\begin{eqnarray}
		& \,\,\, & |e^{-\varpi \langle v \rangle t }  \nabla_v  f(t,x,v)|  \nonumber
		\\
		& \le &   \mathbf{1}_{ \{ \tb(t,x,v)> t \}}  
		 e^{ -\int_0^t  - \frac{C}{2}\langle V(\tau) \rangle d\tau } |\nabla_v f(0,X(0;t,x,v), V(0;t,x,v))|\label{g_initial}\\
		& + &   \ \mathbf{1}_{ \{ \tb(t,x,v)<t \} }
		 e^{-\varpi \langle \vb \rangle \tb} \mu(\vb)^{\frac{1}{4}}  \int_{n(\xb) \cdot u>0} 
		| f(t-\tb, \xb, u) |\sqrt{\mu} \{n(\xb) \cdot u\} \dd u\label{g_bdry}\\
		&  +&   \int^t_{\max\{t-\tb, 0\}} 
		 e^{ -\int_s^t  - \frac{\varpi}{2}\langle V(\tau) \rangle d\tau } e^{-\varpi \langle V(s) \rangle s }  
		  |\nabla_x f(s, X(s),V(s))|
		\dd s\label{g_x}\\
		&   + &\int^t_{\max\{t-\tb, 0\}} \label{g_Gamma}
		(1+ \| w_{\vartheta'} f \|_\infty  )
		e^{ -\int_s^t  - \frac{\varpi}{2}\langle V(\tau) \rangle d\tau }  e^{-\varpi \langle V(s) \rangle s } 
		\\ \notag && \,\,\,  \times \int_{\R^3} \frac{e^{-C_{\vartheta'} |V(s) - u |^2 }}{ |V(s) - u | ^{2 -\kappa } } \nabla_v f(s,X(s),u)| \dd u 
		\dd s\label{g_K}\\
		& +  & \label{nablaphigint}
		\| w_{\vartheta'} f\|_\infty \int^t_{\max\{t-\tb, 0\}} 
		  e^{ -\int_s^t  - \frac{\varpi}{2}\langle V(\tau) \rangle d\tau } e^{-\varpi \langle V(s) \rangle s } e^{-\vartheta' |V(s) |^2 }
		  \\&& \notag \times |\nabla_x \phi (s, X(s;t,x,v))| 
		\dd s. \label{g_phi}
		\end{eqnarray}
%
We first have
		\Be\begin{split}\label{est_g_initial}
			&\| (\ref{g_initial})\|_{L^3_x L^{1+ \delta}_v}\\
			\lesssim & \left(
			\int_{\O}
			\left(\int_{\R^3} |e^{\vartheta' |V(0) | ^2 } \nabla_v f(0,X(0 ), V(0 ))|^3 
			\right)
			\left(
			\int_{\R^3} e^{-(1+ \delta) \frac{3}{2-\delta} \vartheta' |V(0) | ^2 } \dd v \right)^{\frac{2-\delta}{1+ \delta}}
			\right)^{1/3} \\
			\lesssim & \
			\left(\iint_{\O \times \R^3} |e^{\vartheta' |V(0) | ^2 } \nabla_v f(0,X(0;t,x,v), V(0;t,x,v))|^3 \dd v \dd x\right)^{1/3}
			\\
		 \lesssim & \ \| w_{\vartheta'} \nabla_v f (0) \|_{L^3_{x,v}},
		\end{split}
		\Ee
		where we have used a change of variables $(x,v) \mapsto (X(0;t,x,v), V(0;t,x,v))$.

		\hide
		Also we use $|V(0;t,x,v)| \gtrsim |v|$ for $|v|\gg 1$, from (\ref{decay_phi}), and hence $\tilde{w}(V(0;t,x,v))^{- (1+ \delta) \frac{3}{2-\delta}} \in L_v^1 (\R^3)$.\unhide
		
		Clearly 
		\Be\label{est_g_bdry}
		\|(\ref{g_bdry})\|_{L^3_x L^{1+ \delta}_v}  \lesssim \sup_{0 \leq s \leq t} \| w_{\vartheta'} f (s) \|_\infty.
		\Ee
		
		From $ \| \nabla_x \phi \|_{L^3} \lesssim \| \phi \|_{W^{2,2}_x}$ for a bounded $\O \subset \R^3$, and the change of variables $(x,v) \mapsto (X(s;t,x,v), V(s;t,x,v))$ for fixed $s\in(\max\{t-\tb,0\},t)$,
		\Be
		\begin{split}\label{est_g_phi}
			\|  (\ref{nablaphigint})\|_{L^3_x L^{1+ \delta}_v} 
			 \lesssim  &  \| w_{\vartheta'} f\|_\infty \int^t_{\max\{t-\tb,0\}} \| \phi (s) \|_{W^{2,2}_{x} } 
			\\
			\lesssim & \ \| w_{\vartheta'} f\|_\infty \int^t_{\max\{t-\tb,0\}} \| \int_{\mathbb R^3} \sqrt \mu f(s) dv  \|_{2}.
			 \lesssim  t  \| w_{\vartheta'} f\|_\infty \| w_{\vartheta'} f \|_\infty .
		\end{split}
		\Ee

%
%

		
		Next we have from (\ref{int1overalphadv}), for $\frac{3\delta}{ 2 (1+\delta) } < 1$, equivalently $0 < \delta < 2 $,
		{
				\Be\label{init_p_xf}
		\begin{split}
			 \|(\ref{g_x})\|_{L^3_x L^{1+ \delta}_v}  \le &\left\|\left\| \int^t_{\max\{t-\tb, 0\}}
			\nabla_x f(s,X(s),V(s)) \dd s
			\right\|_{L_{v}^{1+ \delta}(  \R^3)}\right\|_{L^{3}_x}\\
			= & \ \left\|\left\| \int^t_{\max\{t-\tb, 0\}}    \frac{ e^{\vartheta' |V(s) |^2 } e^{-\varpi \langle V(s) \rangle s} \tilde \alpha \nabla_x f(s,X(s),V(s))}{e^{\vartheta' |V(s) |^2 } e^{-\varpi \langle V(s) \rangle s} \tilde \alpha}
			\dd s
			\right\|_{L_{v}^{1+ \delta}(  \R^3)}\right\|_{L^{3}_x}
			\\
			\le & \sup_{0 \le t \le T} \ \left\|  w_{\vartheta'} e^{-\varpi \langle v \rangle t} \tilde \alpha \nabla_x f \right\|_\infty 
			\\
			& \times \left\|
			 \left\| \int^t_{\max\{t-\tb, 0\}}  \frac{e^{- \vartheta' |V(s) |^2 } e^{\varpi \langle V(s) \rangle s }}{\tilde \alpha(s, X(s), V(s) )}		\dd s
\right\|_{L_{v}^{1+\delta}( \R^3)}\right\|_{L^{3}_x}\\
			\lesssim &   e^{C ( \| \nabla \phi \|_\infty + \| \nabla \phi \|_\infty^2 +\| \nabla ^2 \phi \|_\infty) }
 \sup_{0 \le t \le T} \ \left\|  w_{\vartheta'} e^{-\varpi \langle v \rangle t } \tilde \alpha \nabla_x f \right\|_\infty  
					\\ & \times  t  \int_{\Omega }  \left( \int_{\mathbb R^3 }   \frac{e^{- \frac{\vartheta'}{2} |v |^2 }}{(\tilde \alpha (t, x, v ))^{1+\delta}}  \dd v \right)^{\frac{3}{1+\delta}}  \dd x			
		\\ \lesssim & t e^{C (  \| \nabla \phi \|_\infty^2 +\| \nabla ^2 \phi \|_\infty) }  \sup_{0 \le t \le T} \ \left\|  w_{\vartheta'} e^{-\varpi \langle v \rangle t } \tilde \alpha \nabla_x f \right\|_\infty.
		\end{split}\Ee
		}
		 		
Next, we consider (\ref{g_Gamma}). From the computations in (\ref{int1overalphadv}), and using the fact that $\frac{1}{\tilde \alpha} \lesssim \frac{1}{\tilde \alpha^\beta}$, we have
{
\begin{equation} \label{nonlocall3l1est} \begin{split}
 & \|(\ref{g_Gamma})\|_{L^3_x L^{1+ \delta}_v}  
 \\ \le & \left\|\left\| \int^t_{\max\{t-\tb, 0\}}e^{ -\int_s^t  - \frac{\varpi}{2}\langle V(\tau) \rangle d\tau }  e^{-\varpi \langle V(s) \rangle s } \right. \right. 
\\ & \left. \left.  \quad \quad \quad \quad   \quad \quad \quad \quad \times  \int_{\R^3} \frac{e^{-C_{\vartheta'} |V(s) - u |^2 }}{ |V(s) - u | ^{2 -\kappa } } \nabla_v f(s,X(s),u)| \dd u \dd s
			\right\|_{L_{v}^{1+ \delta}(  \R^3)}\right\|_{L^{3}_x}
\\ \lesssim &e^{C \| \nabla \phi \|_\infty }   \sup_{0 \le t \le T} \ \left\|  w_{\vartheta'} e^{-\varpi \langle v \rangle t } \tilde \alpha \nabla_x f \right\|_\infty 
\\ & \times \left\|\left\| \int^t_{\max\{t-\tb, 0\}}e^{ -\int_s^t  - \frac{\varpi}{2}\langle V(\tau) \rangle d\tau } 
 \int_{\R^3} \frac{e^{-C_{\vartheta'} |V(s) - u |^2 }}{ |(s) - u | ^{2 -\kappa } } \frac{e^{- \frac{\vartheta'}{2} |u|^2 }  }{(\tilde \alpha(s,X(s),u))^\beta} \dd u \dd s
			\right\|_{L_{v}^{1+ \delta}(  \R^3)}\right\|_{L^{3}_x}.
\end{split}
\Ee
And then applying the nonlocal-to-local estimate \eqref{nonlocaltolocal} to \eqref{nonlocall3l1est} , we conclude

\Be
\begin{split}
 & \|(\ref{g_Gamma})\|_{L^3_x L^{1+ \delta}_v}   
 \\ \lesssim &e^{C (\| \nabla \phi \|_\infty   +  \| \nabla^2 \phi \|_\infty)}  \sup_{0 \le t \le T} \ \left\|  w_{\vartheta'} e^{- \varpi \langle v \rangle t } \tilde \alpha \nabla_x f \right\|_\infty 
\\ & \times \left\|\left\|
\frac{  \delta ^{\frac{3 -\beta}{2} }}{(\tilde \alpha(t,x,v))^{\beta -2 }  (|v|^2 + 1 )^{\frac{3 -\beta}{2}}}  + \frac{ (|v| +1)^{\beta -1} } { \delta ^{\beta -1 } \varpi \langle v \rangle  ( \tilde \alpha ( t,x,v) )^{\beta - 1 } }			\right\|_{L_{v}^{1+ \delta}(  \R^3)}\right\|_{L^{3}_x}
\\ \lesssim &e^{C (\| \nabla \phi \|_\infty   +  \| \nabla^2 \phi \|_\infty)}  \sup_{0 \le t \le T} \ \left\|  w_{\vartheta'} e^{-\varpi \langle v \rangle t } \tilde \alpha \nabla_x f \right\|_\infty 
\\ & \times \left( O(\delta^{\frac{3-\beta}{2}} )   + \frac{1}{ \delta ^{\beta -1 } \varpi}  \left\|\left\|
   \frac{ 1} { \langle v \rangle^{2 -\beta}  ( \tilde \alpha ( t,x,v) )^{\beta - 1 } }			\right\|_{L_{v}^{1+ \delta}(  \R^3)}\right\|_{L^{3}_x} \right)
\\ \lesssim & C(\delta^{\frac{3-\beta}{2}}  +  \frac{1}{ \delta ^{\beta -1 } \varpi}  )e^{C (\| \nabla \phi \|_\infty  +  \| \nabla^2 \phi \|_\infty)}  \sup_{0 \le t \le T} \ \left\|  w_{\vartheta'} e^{-\varpi \langle v \rangle t } \tilde \alpha \nabla_x f \right\|_\infty,
\end{split} \end{equation}
}
for $\beta$ satisfies $\frac{  (\beta-1)(1+\delta) -1}{2} \frac{3}{1+\delta} < 1$, which is equivalent to $\beta< \frac{5}{3} + \frac{1}{1+\delta}$. Therefore any $1 < \beta < \frac{5}{3}$ would work.

		Collecting terms from (\ref{g_initial})-(\ref{nablaphigint}), and (\ref{est_g_initial}), (\ref{est_g_bdry}), (\ref{est_g_phi}), (\ref{init_p_xf}), (\ref{nonlocall3l1est}), we derive
		\Be\begin{split}\label{bound_nabla_v_g}
			& \sup_{0 \leq s \leq t}\| e^{-\varpi \langle v \rangle t }  \nabla_vf(s) \|_{L^3_xL^{1+ \delta}_v} \\
	\lesssim & \| w_{\vartheta'} \nabla_v f (0) \|_{L^3_{x,v}} +  \| w_{\vartheta'} f \|_\infty )^2 +  \| w_{\vartheta'} f \|_\infty
	\\ < & \infty.
			\end{split} \Ee
This proves (\ref{L3L1plusbddsolution}) and conclude Theorem \ref{WlinftyVPBthm}.

\end{proof}

\section{Acknowledgements}This work was supported in part by National Science Foundation under Grant No. 1501031, Grant No. 1900923, and the Wisconsin Alumni Research Foundation.

\end{document}